\documentclass[10pt,leqno]{amsart}
\usepackage{graphicx}
\usepackage{indentfirst,csquotes}

\usepackage{lipsum} % For filler text

\usepackage{amssymb,amsthm,amsmath}
\usepackage{xcolor,paralist,hyperref,titlesec,fancyhdr,etoolbox}
\newtheorem{theorem}{Theorem}[]
\newtheorem{definition}[theorem]{Definition}

\newtheorem{lemma}[theorem]{Lemma}
\newtheorem{remark}[theorem]{Remark}
\newtheorem{proposition}[theorem]{Proposition}

\titleformat{\section}[display]{\normalfont\huge\bfseries\centering}{\centering}{10pt}{\Large}
\titlespacing*{\section}{0pt}{0ex}{0ex}

\usepackage{titlesec}

\titleformat{\subsection}[block] % block = left-aligned, inline label and title
  {\normalfont\large\bfseries}   % font: normal, large, bold
  {\thesubsection}               % subsection label (e.g., 1.1)
  {1em}                          % space between label and title
  {}                             % before-code (nothing extra)

\titlespacing*{\subsection}
  {0pt}     % left margin (no indent)
  {1.5ex plus 1ex minus .2ex}  % space before subsection
  {1ex}     % space after subsection
\hypersetup{ colorlinks=true, linkcolor=black, filecolor=black, urlcolor=black }

\usepackage{lipsum}

\begin{document}
\title{On the optimality conditions for a fractional diffusive equation with a nonlocal term} %%%%%%%%%%%%
\author[1]{Jasarat Gasimov}
\author[2]{Nazim Mahmudov}

\address[1]{Department of Mathematics, Eastern Mediterranean University, Mersin 10, 99628, T.R. North Cyprus, Turkey\\
\texttt{jasarat.gasimov@emu.edu.tr}}
\address[2]{Department of Mathematics, Eastern Mediterranean University, Mersin 10, 99628, T.R. North Cyprus, Turkey\\
\texttt{nazim.mahmudov@emu.edu.tr}}

\date{\today}
%\address{Department of Mathematics, Eastern Mediterranean University, Mersin 10, 99628, T.R. North Cyprus, Turkey}
%\email{jasarat,gasimov@emu.edu.tr}
\maketitle

\let\thefootnote\relax
%\footnotetext{MSC2020: Primary 00A05, Secondary 00A66.}

\begin{abstract}
We study a bilinear OCP for an evolution equation governed by the fractional Laplacian of order $0 < s < 1$, incorporating a nonlocal time component modeled by an integral kernel. After establishing well-posedness of the problem, we analyze the properties of the control-to-state operator. We prove the existence of at least one optimal control and derive both first-order and second-order optimality conditions, which ensure local uniqueness. Under further assumptions, we also demonstrate that global uniqueness of the optimal control can be achieved.
\end{abstract} %%%%%%%%%

\bigskip




\maketitle

\section{Introduction}\label{sec1}
Let $\Omega \subset \mathbb{R}^N$ be a bounded domain with boundary $\partial\Omega$, where $N \geq 1$, and let $\omega \subset \Omega$ be a non-empty open subset. Given parameters $T > 0$ and $\alpha > 0$, along with a desired state $y^d \in L^{\infty}(\Omega)$, we consider the following optimal control problem (OCP).
 Find
\begin{align}\label{d1}
   \inf_{v\in U}J(v):=\frac{1}{2}\Vert y(\cdot,T)-y^d\Vert^{2}_{L^{2}(\Omega)}+\frac{\alpha}{2}\Vert v\Vert^{2}_{L^{2}(\omega\times(0,T))},
\end{align}
subject to the constraints that $y$ solves the space fractional diffusion equation
\begin{align}\label{e}
\begin{cases}
     y_{t}+(-\Delta)^{s}y(t,x)+\int_{0}^{t}\kappa(t,\tau,x)y(\tau,x)d\tau=vy\chi_{\omega}+f(t,x)\quad \text{in} \quad Q:=\Omega\times(0,T)\\
    y=0\quad\quad\text{in}\quad\Sigma:=,(R^{n}\setminus\Omega)\times(0,T)\\
    y(\cdot,0)=y_0\quad\text{in}\quad\Omega
\end{cases}
\end{align}
and the set of admissible controls is given by 
\begin{align}
    U := \left\{ v \in L^{\infty}(\omega \times (0,T)) : m \leq v \leq M \right\},
\end{align}
where \( m, M \in \mathbb{R} \) with \( M > m \). In \eqref{d1}, \( (-\Delta)^s \) represents the fractional Laplace operator of order \( 0 < s < 1 \), \( y_0 \in L^{\infty}(\Omega) \) is the initial condition, and \( \chi_{\omega} \) denotes the characteristic function of the set \( \omega \).

In recent years, the fractional Laplacian has gained considerable attention in the literature; see, for example, \cite{RZ,GSG,kmw,dpv}. Optimal control problems (OCPs) involving partial differential equations (PDEs) have also been extensively investigated, as discussed in \cite{tf, tf1, at, ct, ct1, ct2}. In particular, there has been increasing interest in bilinear optimal control of PDEs, with relevant contributions found in \cite{mkw, kmw, kdz, kdm, kdl}.
 The first study on bilinear optimal control for fractional diffusion equations was conducted by Mophou, Kenne, and Warma in \cite{mkw}. Inspired by their work, we consider a related problem in this paper. However, our study differs from \cite{mkw} by including a Volterra-type nonlocal term of the form 
\[
\int_{0}^{t} \kappa(t,\tau,x)y(\tau,x)\,d\tau.
\]

\textbf{Assumption:}\quad
\(\kappa(\cdot,\cdot,\cdot)\in L^{\infty}(Q\times[0,T]), \quad \kappa\ge0.\)

We first present the main results related to the extended model. Before discussing these results, we provide some general background on the fractional Laplacian in Section~2.

Section~3 is divided into five subsections:

\begin{itemize}
    \item In \textbf{Subsection~3.1}, we prove the existence and uniqueness of the solution to equation~\eqref{e}, along with some important estimates that will be useful in the following parts. To do this, we apply the transformation \( z(t,x) = e^{-rt} y(t,x) \), which leads to equation~\eqref{d} (see~\cite{bekt}). Instead of working directly with equation~\eqref{e}, we study equation~\eqref{d} and prove existence and uniqueness using standard methods.

    \item In \textbf{Subsection~3.2}, we prove the maximum principle, which gives important information about the behavior of the solution.

    \item In \textbf{Subsection~3.3}, we show that there exists an optimal solution to the control problem.

    \item In \textbf{Subsections~3.4} and \textbf{3.5}, we present the first-order necessary and second-order  conditions for optimality, which help describe the properties of the optimal solution.
\end{itemize}

\section{Preliminaries}

\subsection{Mathematical background}
To ensure clarity and completeness, we recall some standard results that are used throughout the paper. We begin by defining the fractional Laplace operator. For a parameter \( 0 < s < 1 \), we define the space
\begin{align*}
    L^1_s(\mathbb{R}^N) := \left\{ \vartheta  : \mathbb{R}^N \to \mathbb{R} \ \text{measurable s.t.} \ \int_{\mathbb{R}^N} \frac{|\vartheta (x)|}{(1 + |x|)^{N+2s}} \, dx < \infty \right\}.
\end{align*}

This space consists of measurable functions on \( \mathbb{R}^N \) whose integrals, weighted by \( (1 + |x|)^{-(N+2s)} \), are finite.

For any function \( \vartheta  \in L^1_s(\mathbb{R}^N) \) and for every \( \varepsilon > 0 \), we define the truncated fractional Laplacian as
\[
(-\Delta)^s_\varepsilon\vartheta (x) := C_{N,s} \int_{\{ y \in \mathbb{R}^N : |x - y| > \varepsilon \}} \frac{\vartheta (x) - \vartheta (y)}{|x - y|^{N + 2s}} \, dy, \quad x \in \mathbb{R}^N,
\]
where \( C_{N,s} \) is a normalization constant given by
\[
C_{N,s} := \frac{2^{2s} \Gamma\left( \frac{N + 2s}{2} \right)}{\pi^{N/2} |\Gamma(-s)|}.
\]

The fractional Laplacian \( (-\Delta)^s \) is then defined through the following singular integral:
\begin{equation} \label{f4}
(-\Delta)^s \vartheta (x) := C_{N,s} \, \text{P.V.} \int_{\mathbb{R}^N} \frac{\vartheta (x) - \vartheta (y)}{|x - y|^{N + 2s}} \, dy = \lim_{\varepsilon \to 0} (-\Delta)^s_\varepsilon \vartheta (x), \quad x \in \mathbb{R}^N,
\end{equation}
whenever this limit exists for almost every \( x \in \mathbb{R}^N \). For alternative characterizations of the fractional Laplacian, we refer the reader to \cite{dpv,km}.

We now introduce the functional framework required for our analysis. Let \( \Omega \subset \mathbb{R}^N \) be an arbitrary open set, with \( N \geq 1 \), and let \( 0 < s < 1 \). The fractional Sobolev space of order \( s \) is defined as
\[
H^s(\Omega) := \left\{\vartheta  \in L^2(\Omega) : \int_{\Omega} \int_{\Omega} \frac{|\vartheta (x) - \vartheta (y)|^2}{|x - y|^{N + 2s}} \, dx \, dy < \infty \right\},
\]
and is equipped with the norm
\[
\|\vartheta \|_{H^s(\Omega)} := \left( \int_{\Omega} |\vartheta (x)|^2 \, dx + \int_{\Omega} \int_{\Omega} \frac{|\vartheta (x) - \vartheta (y)|^2}{|x - y|^{N + 2s}} \, dx \, dy \right)^{1/2}.
\]

We define the space
\[
H^s_0(\Omega) := \left\{ \vartheta  \in H^s(\mathbb{R}^N) : \vartheta  = 0 \text{ a.e. in } \mathbb{R}^N \setminus \Omega \right\}.
\]
This space is endowed with the norm
\begin{equation}\label{f3}
\|\vartheta \|_{H^s_0(\Omega)} := \left( \frac{C_{N,s}}{2} \int_{\mathbb{R}^N} \int_{\mathbb{R}^N} \frac{(\vartheta (x) - \vartheta (y))^2}{|x - y|^{N + 2s}} \, dx \, dy \right)^{1/2},
\end{equation}
under which \( H^s_0(\Omega) \) becomes a Hilbert space (see, for instance, [\cite{sv}, Lemma 7]).

The dual space of \( H^s_0(\Omega) \) with respect to the pivot space \( L^2(\Omega) \) is denoted by
\[
H^{-s}(\Omega) := \left( H^s_0(\Omega) \right)',
\]
and we have the continuous embeddings (see, e.g., \cite{28}):
\begin{equation}\label{f2}
H^s_0(\Omega) \hookrightarrow L^2(\Omega) \hookrightarrow H^{-s}(\Omega).
\end{equation}

For any \( \rho, \psi \in H^s_0(\Omega) \), we define the bilinear form
\begin{equation}\label{f1}
\mathcal{M}(\vartheta , \psi) := \frac{C_{N,s}}{2} \int_{\mathbb{R}^N} \int_{\mathbb{R}^N} \frac{(\vartheta (x) - \vartheta (y))(\psi(x) - \psi(y))}{|x - y|^{N + 2s}} \, dx \, dy.
\end{equation}
In particular, the norm in \eqref{f3} can be expressed as \( \|\vartheta \|_{H^s_0(\Omega)} = \left( \mathcal{M}(\vartheta , \vartheta ) \right)^{1/2} \).

We now define the operator \( (-\Delta)^s_D \) on \( L^2(\Omega) \) by
\begin{equation}\label{fc}
D((-\Delta)^s_D) := \left\{ \vartheta  \in H^s_0(\Omega) : (-\Delta)^s \vartheta  \in L^2(\Omega) \right\}, \quad (-\Delta)^s_D \vartheta  := (-\Delta)^s \vartheta  \quad \text{in } \Omega.
\end{equation}
This corresponds to the realization of the fractional Laplacian \( (-\Delta)^s \) in \( L^2(\Omega) \), subject to the homogeneous Dirichlet condition imposed outside \( \Omega \).

The following result is classical and can be found, for instance, in \cite{RZ, GSG}.
 
\begin{proposition}\label{p}
\textit{Let \( (-\Delta)^s_D \) be the operator defined in \eqref{fc}. Then \( (-\Delta)^s_D \) can also be interpreted as a bounded operator from \( H^s_0(\Omega) \) into \( H^{-s}(\Omega) \), defined via the duality pairing}
\begin{align}\label{fc1}
\langle (-\Delta)^s_D u, v \rangle_{H^{-s}(\Omega), H^s_0(\Omega)} := \mathcal{M}(u, v), \quad \text{for all } u, v \in H^s_0(\Omega).
\end{align}
\end{proposition}
 
Let \( \mathbb{X} \) be a Banach space with dual space \( \mathbb{X}^\star \). We define the function space
\begin{align}\label{c3}
W(0,T;\mathbb{X}) := \left\{ \psi \in L^2((0,T); \mathbb{X}) : \psi_t \in L^2((0,T); \mathbb{X}^\star) \right\},
\end{align}
which becomes a Hilbert space when endowed with the norm
\begin{align}\label{c4}
\|\psi\|^2_{W(0,T;\mathbb{X})} = \|\psi\|^2_{L^2(0,T;\mathbb{X})} + \|\psi_t\|^2_{L^2(0,T;\mathbb{X}^\star)}.
\end{align}

Now suppose \( \mathbb{Y} \) is a Hilbert space that coincides with its dual, i.e., \( \mathbb{Y} = \mathbb{Y}^\star \), and that the embeddings
\[
\mathbb{X} \hookrightarrow \mathbb{Y} = \mathbb{Y}^\star \hookrightarrow \mathbb{X}^\star
\]
are continuous. Then, by Theorem 1.1 in \cite{jl}, p.~102, it follows that the space \( W(0,T;\mathbb{X}) \) is continuously embedded into \( C([0,T]; \mathbb{Y}) \), that is,
\begin{align}\label{c5}
W(0,T;\mathbb{X}) \hookrightarrow C([0,T]; \mathbb{Y}).
\end{align}

In addition, by applying the compactness result from \cite{jl1}, we obtain the compact embedding
\begin{align}\label{k}
W(0,T;\mathbb{X}) \subset L^2(Q).
\end{align}

We are now in a position to prove the following preliminary result.
 
\begin{lemma}\label{lem:strong-conv}
Let \( (z_n)_n \) be a sequence in \( W(0,T; \mathbb{X}) \) s.t.
\[
z_n \rightharpoonup z \quad \text{weakly in } W(0,T; \mathbb{X}) \quad \text{as } n \to \infty.
\]
Then, the following assertions hold:
\begin{align}\label{k1}
z_n \to z \quad \text{strongly in } L^2(Q) \quad \text{as } n \to \infty,
\end{align}
and
\begin{align}\label{k2}
\lim_{n \to \infty} \int_Q \left( \int_0^t e^{r(\tau - t)} \kappa(t, \tau, x) z_n(\tau, x) \, d\tau \right) dx \, dt\nonumber\\
= \int_Q \left( \int_0^t e^{r(\tau - t)} \kappa(t, \tau, x) z(\tau, x) \, d\tau \right) dx \, dt.
\end{align}
\end{lemma}
 
\begin{proof}
The strong convergence in \eqref{k1} follows directly from the compact embedding \( W(0,T; \mathbb{X}) \hookrightarrow L^2(Q) \), stated in \eqref{k}.

To prove \eqref{k2}, observe that the strong convergence \eqref{k1} implies that
\[
z_n \rightharpoonup z \quad \text{weakly in } L^2(Q) \quad \text{as } n \to \infty.
\]
Since \( \kappa\in L^\infty(Q \times [0,T]) \), we have for each \( t \in [0,T] \), the function
\[
(t,x) \mapsto \int_0^t e^{r(\tau - t)} \kappa(t, \tau, x) z_n(\tau, x) \, d\tau
\]
is well-defined and belongs to \( L^2(Q) \), uniformly in \( n \). Using Fubini’s theorem and the strong convergence in \( L^2(Q) \), we obtain:
\begin{align*}
&\lim_{n \to \infty} \int_Q \left( \int_0^t e^{r(\tau - t)} \kappa(t, \tau, x) z_n(\tau, x) \, d\tau \right) dx \, dt \\
&= \lim_{n \to \infty} \int_Q z_n(\tau, x) \left( \int_\tau^T e^{r(\tau - t)} \kappa(t, \tau, x) \, dt \right) dx \, d\tau \\
&= \int_Q z(\tau, x) \left( \int_\tau^T e^{r(\tau - t)} \kappa(t, \tau, x) \, dt \right) dx \, d\tau \\
&= \int_Q \left( \int_0^t e^{r(\tau - t)} \kappa(t, \tau, x) z(\tau, x) \, d\tau \right) dx \, dt.
\end{align*}
This establishes the identity \eqref{k2}.
\end{proof}
 
\begin{remark}
If \( z, \phi \in L^2((0,T); \mathbb{X}) \), then for every \( t \in [0,T] \), the following estimate holds:
\begin{align}\label{est}
\left| \int_Q \left( \int_0^t e^{r(\tau - t)} \kappa(t, \tau, x) z(\tau, x) \, d\tau \right) \phi(t, x) \, dx \, dt \right|
\lesssim \|\kappa\|_{L^\infty} \|z\|_{L^2(Q)} \|\phi\|_{L^2(Q)}.
\end{align}
\noindent
This estimate follows from Hölder’s inequality and the boundedness of \( \kappa\), with constants depending on \( |Q|^{1/2} \) and \( T^{1/2} \).
\end{remark}

\section{Main results}
\subsection{Existence and uniqueness}\label{sec2}
Throughout the rest of the paper, let \( \Omega \subset \mathbb{R}^{n} \) be a bounded domain. For a nonempty open set \( \omega \subset \Omega \) and \( T > 0 \), we define
\[
\omega_T := \omega \times (0,T),
\]
and we denote by \( \|\cdot\|_{\infty} \) the norm in \( L^{\infty}(\omega_T) \). 

In addition, for \( u, \omega \in H_0^s(\Omega) \), let \( \mathcal{M}(u,\omega) \) denote the bilinear form defined in \eqref{f1}. To simplify notation, we set
\begin{align}
   \mathbb{V} := H_0^s(\Omega), \qquad \mathbb{V}^* := H^{-s}(\Omega),
\end{align}
and denote by \( \langle \cdot, \cdot \rangle_{\mathbb{V}^*, \mathbb{V}} \) the duality pairing between \( \mathbb{V}^* \) and \( \mathbb{V} \).

Let \( r > 0 \) be a given real number. We consider the following nonlocal-in-time system:
\begin{align}\label{d}
\begin{cases}
      z_t + (-\Delta)^s z(t,x) + \int_0^t e^{r(\tau - t)} \kappa(t,\tau,x) z(\tau,x) \, d\tau + r z \\
    = v z \chi_{\omega} + e^{-r t} f \quad \text{in } Q , \\
    z = 0 \qquad \text{on } \Sigma , \\
    z(\cdot, 0) = y_0 \quad \text{in } \Omega. 
\end{cases}
\end{align}

\begin{definition}[Weak solution]\label{wd}
Let \( f \in L^2((0,T); \mathbb{V}^*) \), \( v \in L^\infty(\omega_T) \), and \( y_0 \in L^2(\Omega) \). We say that \( z \in L^2((0,T); \mathbb{V}) \) is a \emph{weak solution} of \eqref{d} if for every test function \( \phi \in \mathbb{H}(Q) \), the following identity holds:
\begin{align*}
    - \int_0^T \langle \phi_t, z \rangle_{\mathbb{V}^*, \mathbb{V}} \, dt 
    + \int_0^T \mathcal{M}(z, \phi) \, dt 
    + \int_Q \left( \int_0^t e^{r(\tau - t)} \kappa(t,\tau,x) z(\tau,x) \, d\tau \right) \phi(t,x) \, dxdt \\
    + r \int_Q z \phi \, dxdt 
    - \int_{\omega_T} v z \phi \, dxdt 
    = \int_0^T e^{-rt} \langle f(t), \phi(t) \rangle_{\mathbb{V}^*, \mathbb{V}} \, dt 
    + \int_\Omega y_0 \phi(0,x) \, dx.
\end{align*}
Here, the space of test functions is defined by
\[
\mathbb{H}(Q) := \left\{ \varphi \in W(0,T; \mathbb{V}) : \varphi(\cdot, T) = 0 \ \text{a.e. in } \Omega \right\}.
\]
\end{definition}

\begin{theorem}\label{the}
    Let $f \in L^{2}((0,T);\mathbb{V}^*)$, $v \in L^{\infty}(\omega_T)$, and define
    \[
        r := \|v\|_{\infty} + \|\kappa\|_{\infty} + 1,
    \]
    with initial data $y_0 \in L^{2}(\Omega)$. Then, there exists a unique weak solution
    \[
        z \in W(0,T;\mathbb{V})
    \]
    of problem \eqref{d}. Moreover, there exists a constant $C = C(N,s,\Omega) > 0$ s.t. the following estimates hold:
    \begin{align}\label{1the}
        \sup_{\tau \in [0,T]} \|z(\cdot,\tau)\|_{L^{2}(\Omega)}^2 &\leq \|f\|_{L^{2}((0,T);\mathbb{V}^*)}^2 + \|y_0\|_{L^{2}(\Omega)}^2, \nonumber \\
        \|z\|_{L^{2}((0,T);\mathbb{V})}^2 &\leq \|f\|_{L^{2}((0,T);\mathbb{V}^*)}^2 + \|y_0\|_{L^{2}(\Omega)}^2, \\
        \|z\|_{W(0,T;\mathbb{V})} &\leq \bigl(C \|v\|_{\infty} + \|\kappa\|_{\infty} + 3 \bigr) \bigl(\|f\|_{L^{2}((0,T);\mathbb{V}^*)} + \|y_0\|_{L^{2}(\Omega)}^2 \bigr). \nonumber
    \end{align}
\end{theorem}

\begin{proof}
    The proof proceeds in four main steps.

    \textbf{Step 1: Existence.} We use Theorem 1.1 in \cite{jl2} to establish existence. Recall that the norm on $L^{2}((0,T);\mathbb{V})$ is defined by
    \[
        \|z\|_{L^{2}((0,T);\mathbb{V})}^2 = \int_0^T \|z(\cdot,t)\|_{\mathbb{V}}^2 \, dt.
    \]
    Define the space
    \[
        H(Q) := \left\{ z : \|z\|_{H(Q)}^2 := \|z\|_{L^{2}((0,T);\mathbb{V})}^2 + \|z(\cdot,0)\|_{L^{2}(\Omega)}^2 < \infty \right\}.
    \]
    We observe that the embedding $H(Q) \hookrightarrow L^{2}((0,T);\mathbb{V})$ is continuous.

    For any $\phi \in H(Q)$, consider the bilinear form
    \begin{align*}
         B(z,\phi) := -\int_0^T \langle \phi_t, z \rangle_{\mathbb{V}^*, \mathbb{V}} dt + \int_0^T \mathcal{M}(z,\phi) dt + \int_Q \left( \int_0^t e^{r(\tau - t)} \kappa(t,\tau,x) z(\tau,x) d\tau \right) \phi \, dx dt\\ + r \int_Q z \phi \, dx dt - \int_{\omega_T} v z \phi \, dx dt.
    \end{align*}
    Using the Cauchy–Schwarz inequality and the boundedness of $v$ and $K$, it follows that $B$ is continuous in $z$ for fixed $\phi$.

    Moreover, choosing $r = \|v\|_{\infty} + \|\kappa\|_{\infty} + 1$ ensures coercivity:
    \[
        B(\phi,\phi) \geq \frac{1}{2} \|\phi(\cdot,0)\|_{L^{2}(\Omega)}^2 + \int_0^T \mathcal{M}(\phi,\phi) dt + (r - \|v\|_{\infty} - \|\kappa\|_{\infty}) \int_Q \phi^2 \, dx dt \geq \|\phi\|_{H(Q)}^2.
    \]

    Next, define the linear functional
    \[
        L(\phi) := \int_0^T e^{-r t} \langle f, \phi \rangle_{\mathbb{V}^*, \mathbb{V}} dt + \int_{\Omega} y_0 \phi(x,0) \, dx,
    \]
    which is continuous on $H(Q)$ due to Cauchy–Schwarz’s inequality.

    By applying Theorem 1.1, there exists $z \in L^{2}((0,T);\mathbb{V})$ s.t.
    \[
        B(z,\phi) = L(\phi), \quad \forall \phi \in H(Q),
    \]
    yielding a weak solution to \eqref{d}.

    \textbf{Step 2: Regularity of time derivative.} We rewrite \eqref{d} as
    \begin{align}
        z_t + (-\Delta)^s z + \int_0^t e^{r(\tau - t)} \kappa(t,\tau,x) z(\tau,x) d\tau + r z = v z \chi_{\omega} + e^{-r t} f, \quad z(\cdot,0) = y_0.
    \end{align}
    Since $z \in L^{2}((0,T);\mathbb{V})$, the operator $(-\Delta)^s z(t)$ lies in $\mathbb{V}^*$ for almost every $t$. Also, $v z \chi_{\omega} \in L^{2}(\omega_T) \subset L^{2}((0,T);\mathbb{V}^*)$. Hence,
    \[
        z_t \in L^{2}((0,T);\mathbb{V}^*).
    \]
    The boundedness of the terms on the right-hand side gives the estimate
    \[
        \|z_t\|_{L^{2}((0,T);\mathbb{V}^*)} \leq (C \|v\|_{\infty} + \|\kappa\|_{\infty} + 2) \bigl( \|f\|_{L^{2}((0,T);\mathbb{V}^*)} + \|y_0\|_{L^{2}(\Omega)} \bigr).
    \]
    Thus, $z \in W(0,T;\mathbb{V})$.

    \textbf{Step 3: Energy estimates.} Taking the duality pairing of \eqref{d2} with $z$ and integrating, we obtain
   \begin{align*}
        \frac{1}{2} \frac{d}{dt} \|z(t)\|_{L^{2}(\Omega)}^2 + \mathcal{M}(z(t), z(t)) + r \|z(t)\|_{L^{2}(\Omega)}^2 + \int_{\Omega} \left( \int_0^t e^{r(\tau - t)} \kappa(t,\tau,x) z(\tau,x) d\tau \right) z(t,x) dx\\ = e^{-r t} \langle f(t), z(t) \rangle + \int_{\omega} v(t) z(t)^2 dx.
   \end{align*}
    Using Young’s inequality and the choice of $r$, this leads to
    \[
        \frac{d}{dt} \|z(t)\|_{L^{2}(\Omega)}^2 + \|z(t)\|_{\mathbb{V}}^2 \leq \|f(t)\|_{\mathbb{V}^*}^2.
    \]
    Integrating in time gives the first two bounds in \eqref{1the}. Combining these with Step 2 yields the final estimate.

    \textbf{Step 4: Uniqueness.} Suppose $z_1$ and $z_2$ are two solutions with the same data, and set $\bar{z} = z_1 - z_2$. Then $\bar{z}$ satisfies
    \[
        \bar{z}_t + (-\Delta)^s \bar{z} + \int_0^t e^{r(\tau - t)} \kappa(t,\tau,x) \bar{z}(\tau,x) d\tau + r \bar{z} = v \bar{z} \chi_{\omega}, \quad \bar{z}(\cdot,0) = 0.
    \]
    Testing with $\bar{z}$ and integrating, and using the coercivity and boundedness of $v$, $K$, and the choice of $r$, one concludes
    \[
        \|\bar{z}(\cdot,T)\|_{L^{2}(\Omega)}^2 + \|\bar{z}\|_{L^{2}((0,T);\mathbb{V})}^2 \leq 0,
    \]
    implying $\bar{z} = 0$ a.e. Hence, the solution is unique.
\end{proof}

Through the next definition, we introduce our notion of weak solution for problem \eqref{e}.

\begin{definition}\label{p1}
Let $f \in L^{2}((0,T);\mathbb{V}^{*})$, $v \in L^{\infty}(\omega_{T})$, and $y_0 \in L^{2}(\Omega)$. A function $y \in L^{2}((0,T);\mathbb{V})$ is said to be a \emph{weak solution} of problem \eqref{e} if, for every test function $\phi \in H(Q)$, the following identity holds:
\begin{align*}
    -\int_{0}^{T} \langle \phi_t, y \rangle_{\mathbb{V}^{*}, \mathbb{V}}\,dt 
    + \int_{0}^{T} \mathcal{M}(y, \phi)\,dt 
    + \int_{Q} \left( \int_{0}^{t} \kappa(t,\tau,x)y(\tau,x)\,d\tau \right) \phi\,dxdt
    - \int_{\omega_T} vy\phi\,dxdt \\
    = \int_{0}^{T} \langle f, \phi \rangle_{\mathbb{V}^{*}, \mathbb{V}}\,dt 
    + \int_{\Omega} y_0 \phi(0)\,dx.
\end{align*}
\end{definition}

We then have the following result.

\begin{theorem}\label{t9}
Let $f \in L^{2}((0,T);\mathbb{V}^{*})$, $v \in L^{\infty}(\omega_{T})$, and $y_0 \in L^{2}(\Omega)$. Then, there exists a unique weak solution $y \in W(0,T;\mathbb{V})$ of \eqref{e}. Furthermore, the following estimates hold:
\begin{align}\label{t8}
   & \sup_{\tau \in [0,T]} \| y(\cdot,\tau) \|^2_{L^2(\Omega)} 
    + \| y \|^2_{L^2((0,T);\mathbb{V})}
    \le 2e^{2(\|v\|_{\infty} + \|\kappa\|_{\infty} + 1)T} 
    \left[ \|f\|^2_{L^2((0,T);\mathbb{V}^*)} + \|y_0\|^2_{L^2(\Omega)} \right], \nonumber \\
    &\| y \|_{W(0,T;\mathbb{V})}
    \le 6(\|v\|_{\infty} + \|\kappa\|_{\infty} + 3) e^{2(\|v\|_{\infty} + \|\kappa\|_{\infty} + 1)T} 
    \left[ \|f\|^2_{L^2((0,T);\mathbb{V}^*)} + \|y_0\|^2_{L^2(\Omega)} \right].
\end{align}
\end{theorem}

\begin{proof}
The existence and uniqueness of the weak solution $y \in W(0,T;\mathbb{V})$ follows from Theorem \ref{the}. To derive estimates \eqref{t8}, we substitute $z = e^{-(\|v\|_{\infty} + \|\kappa\|_{\infty} + 1)t} y$ into \eqref{1the} , obtaining:
\begin{align}\label{l1}
    \sup_{\tau \in [0,T]} \| y(\cdot,\tau) \|^2_{L^2(\Omega)} 
    + \| y \|^2_{L^2((0,T);\mathbb{V})}
    \le 2e^{2(\|v\|_{\infty} + \|\kappa\|_{\infty} + 1)T} 
    \left[ \|f\|^2_{L^2((0,T);\mathbb{V}^*)} + \|y_0\|^2_{L^2(\Omega)} \right].
\end{align}

Additionally, one can estimate
\begin{align*}
   & \int_0^T \left| \langle y_t(t), \phi(t) \rangle_{\mathbb{V}^*, \mathbb{V}} \right| dt 
    \le e^{(\|v\|_{\infty} + \|\kappa\|_{\infty} + 1)T} \\
    \times&\left[ 3(\|v\|_{\infty} + \|\kappa\|_{\infty} + 1) \|y\|_{L^2((0,T);\mathbb{V})} 
    + \|f\|_{L^2((0,T);\mathbb{V}^*)} \right] \|\phi\|_{L^2((0,T);\mathbb{V})},
\end{align*}
for all $\phi \in L^2((0,T);\mathbb{V})$. From this, it follows that
\begin{align*}
    \|y_t\|_{L^2((0,T);\mathbb{V}^*)}
    \le 6(\|v\|_{\infty} + \|\kappa\|_{\infty} + 1) e^{(\|v\|_{\infty} + \|\kappa\|_{\infty} + 1)T}
    \left[ \|f\|^2_{L^2((0,T);\mathbb{V}^*)} + \|y_0\|^2_{L^2(\Omega)} \right]^{1/2}.
\end{align*}
Combining this with \eqref{l1} yields the full bound in \eqref{t8}.
\end{proof}

\subsection{Maximum principle}
We introduce the system:
\begin{align}\label{l2}
\begin{cases}
    y_t + (-\Delta)^s y(t,x) + \int_0^t \kappa(t,\tau,x) y(\tau,x)\,d\tau = vy\chi_\omega  &\text{in } Q, \\
    y = 0 &\text{on } \Sigma, \\
    y(\cdot,0) = y_0 &\text{in } \Omega.
\end{cases}
\end{align}
We present some useful consequences of the maximum principle.
 
\begin{lemma}\label{Tam}
    Suppose \( y_0 \in L^2(\Omega) \) satisfies \( y_0 \geq 0 \) almost everywhere in \( \Omega \), and let \( v \in L^{\infty}(\omega_T) \). Then the weak solution \( y \) to \eqref{l2} satisfies \( y \geq 0 \) almost everywhere in \( \mathbb{R}^N \times [0,T] \).
\end{lemma}

\begin{proof}
We express the solution \( y \) as the difference \( y = y^+ - y^- \), where \( y^+ := \max(y, 0) \) and \( y^- := \max(-y, 0) \). To prove that \( y \) is non-negative almost everywhere, it suffices to show that \( y^- = 0 \) almost everywhere in \( \mathbb{R}^N \times [0,T] \).

Note that
\[
y^-=0 \quad \text{on } \Sigma, \quad \text{and} \quad y^-(\cdot,0) = \max(0, -y_0) = 0 \quad \text{a.e. in } \Omega,
\]
since \( y_0 \geq 0 \). Furthermore, \( y^- \in W(0,T; \mathbb{V}) \).

Define the measurable sets
\begin{align}\label{A}
    O^- &:= \left\{ x \in \mathbb{R}^N \;\middle|\; y(x,t) \leq 0 \ \text{for almost every } t \in (0,T) \right\}, \\
    O^+ &:= \left\{ x \in \mathbb{R}^N \;\middle|\; y(x,t) > 0 \ \text{for almost every } t \in (0,T) \right\}.
\end{align}

By testing the weak formulation \eqref{e} with \( \phi = y^-(\cdot,t) \in \mathbb{V} \) and using Proposition~\ref{fc1}, we obtain for a.e. \( t \in [0,T] \):
\begin{align}\label{w}
    \langle y_t(t), y^-(t) \rangle_{\mathbb{V}^*, \mathbb{V}} + \mathcal{M}(y(t), y^-(t)) = \int_{\omega} v(t)y(t)y^-(t)\,dx - \int_{\Omega} \left( \int_0^t \kappa(t,\tau,x) y(\tau,x)\,d\tau \right) y^-(t)\,dx.
\end{align}

Since \( y^-_t = y_t \chi_{O^-} \) and \( y^+ y^- = 0 \), equation~\eqref{w} becomes:
\begin{align}\label{w1}
    \int_{O^-} \frac{d}{dt} |y^-(t)|^2\,dx - \mathcal{M}(y(t), y^-(t)) 
    &= \int_{O^-} v(t) |y^-(t)|^2\,dx \nonumber\\
    &\quad + \int_{\Omega} \left( \int_0^t |\kappa(t,\tau,x)| |y(\tau,x)|\,d\tau \right) |y^-(t)|\,dx.
\end{align}

Now observe that
\begin{align}\label{w2}
    \mathcal{M}(y(t), y^-(t)) = \mathcal{M}(y^+(t), y^-(t)) - \mathcal{M}(y^-(t), y^-(t)).
\end{align}
It is known that \( \mathcal{M}(y^+(t), y^-(t)) \leq 0 \), hence
\[
\mathcal{M}(y(t), y^-(t)) \leq -\mathcal{M}(y^-(t), y^-(t)) \leq 0.
\]
Therefore, from \eqref{w1}, we deduce
\begin{align*}
    \frac{1}{2} \frac{d}{dt} \| y^-(t) \|^2_{L^2(O^-)} 
    &\lesssim \|v\|_{\infty} \| y^-(t) \|^2_{L^2(O^-)} 
    + \|\kappa\|_{\infty} \int_0^t \| y^-(\tau) \|^2_{L^2(O^-)}\,d\tau.
\end{align*}

Applying Gronwall's lemma and using the fact that \( y^-(\cdot,0) = 0 \), we conclude that
\[
\| y^-(\cdot,t) \|^2_{L^2(O^-)} \leq e^{2t(\|v\|_{\infty} + \|\kappa\|_{\infty})} \cdot \| y^-(\cdot,0) \|^2_{L^2(O^-)} = 0,
\]
for all \( t \in [0,T] \). Thus, \( y^- = 0 \) a.e. in \( O^- \times [0,T] \), and hence in \( \mathbb{R}^N \times [0,T] \).

We conclude that \( y \geq 0 \) a.e. in \( \mathbb{R}^N \times [0,T] \), as required.
\end{proof}

We have the following maximum principle.
\begin{theorem}\label{Tam1}
    Let \( y^0 \in L^{\infty}(\Omega) \) and \( v \in L^{\infty}(\Omega) \). Then, the unique weak solution \( y \) of \eqref{e} belongs to \( W(0,T;\mathbb{V}) \cap L^{\infty}(\mathbb{R}^{N} \times (0,T)) \), and satisfies the estimate
    \begin{align}\label{d2}
        \|y\|_{L^{\infty}(\mathbb{R}^{N} \times (0,T))} \leq e^{\|v\|_{\infty}T}\big[ \|y^0\|_{L^{\infty}(\Omega)}+\|f\|_{L^{\infty}(Q)}\big].
    \end{align}
\end{theorem}

\begin{proof}
    For any \((t,x) \in Q\), define
    \[
        z(t,x) := e^{-(\|v\|_{\infty} + \|\kappa\|_{\infty} + 1)t} y(t,x),
    \]
    where \( y \) is the weak solution of \eqref{e}. By Theorem \ref{the}, the function \( z \in W(0,T;\mathbb{V}) \) satisfies the following problem:
    \begin{align*}
        \begin{cases}
            z_t + (-\Delta)^s z + \displaystyle\int_0^t e^{(\|v\|_{\infty} + \|\kappa\|_{\infty} + 1)(\tau - t)} \kappa(t,\tau,x) z(\tau,x) \, d\tau \\
            \quad + (\|v\|_{\infty} + \|\kappa\|_{\infty} + 1) z = v z \chi_{\omega} + e^{-(\|v\|_{\infty} + \|\kappa\|_{\infty} + 1)t} f & \text{in } Q, \\
            z = 0 & \text{on } \Sigma, \\
            z(\cdot,0) = y^0 & \text{in } \Omega.
        \end{cases}
    \end{align*}

    We now claim that
    \[
        z(t,x) \leq \|y^0\|_{L^{\infty}(\Omega)} + \|f\|_{L^{\infty}(Q)} \quad \text{a.e. in } Q.
    \]
    To show this, define
    \[
        w(t,x) := \|y^0\|_{L^{\infty}(\Omega)} + \|f\|_{L^{\infty}(Q)} - z(t,x).
    \]
    Then
    \[
        w(0,x) = \|f\|_{L^{\infty}(Q)} + \left( \|y^0\|_{L^{\infty}(\Omega)} - y^0(x) \right) \geq 0 \quad \text{a.e. in } \Omega.
    \]
    Moreover, \( w \) satisfies
    \begin{align}\label{Ta}
        \begin{cases}
            w_t + (-\Delta)^s w + \displaystyle\int_0^t e^{(\|v\|_{\infty} + \|\kappa\|_{\infty} + 1)(\tau - t)} \kappa(t,\tau,x) w(\tau,x) \, d\tau \\
            \quad + (\|v\|_{\infty} + \|\kappa\|_{\infty} + 1) w = v w \chi_{\omega} + (\|v\|_{\infty} + \|\kappa\|_{\infty} + 1)A \\
            \quad + \left(A - e^{-(\|v\|_{\infty} + \|\kappa\|_{\infty} + 1)t} f \right) & \text{in } Q, \\
            w = \|y^0\|_{L^{\infty}(\Omega)} + \|f\|_{L^{\infty}(Q)} & \text{on } \Sigma, \\
            w(\cdot,0) = \|y^0\|_{L^{\infty}(\Omega)} - y^0 & \text{in } \Omega,
        \end{cases}
    \end{align}
    where \( A := \|y^0\|_{L^{\infty}(\Omega)} + \|f\|_{L^{\infty}(Q)} \). To conclude the result, we prove that \( w^{-} := \max(0,-w) = 0 \) a.e. in \( \mathbb{R}^N \times [0,T] \).

    Note that \( w^- \in W(0,T;\mathbb{V}) \). Define
    \[
        O^- := \{x \in \mathbb{R}^N : w(x,t) \leq 0 \text{ for a.e. } t \in (0,T)\}.
    \]
    By testing \eqref{Ta} with \( w^-(\cdot,t) \), and following the argument in the proof of Lemma \ref{Tam}, we get:
    \begin{align*}
        \frac{1}{2} \frac{d}{dt} \|w^-(t)\|^2_{L^2(O^-)} &\lesssim \|v\|_{\infty} \|w^-(t)\|^2_{L^2(O^-)} + \|\kappa\|_{\infty} \int_0^t \|w^-(\tau)\|^2_{L^2(O^-)} \, d\tau.
    \end{align*}
    Applying Grönwall's inequality and using the fact that \( w^-(\cdot,0) = 0 \) a.e. in \( O^- \), we deduce that
    \[
        \|w^-(\cdot,t)\|_{L^2(O^-)} = 0 \quad \text{for all } t \in [0,T],
    \]
    which implies \( w^- = 0 \) a.e. in \( \mathbb{R}^N \times [0,T] \), i.e., \( w \geq 0 \) a.e. in \( Q \). Therefore,
    \[
        z(t,x) \leq \|y^0\|_{L^{\infty}(\Omega)} + \|f\|_{L^{\infty}(Q)}.
    \]
    Recalling that \( y(t,x) = e^{(\|v\|_{\infty} + \|\kappa\|_{\infty} + 1)t} z(t,x) \).
    %, and absorbing the term \( \|f\|_{L^\infty(Q)} \) into the constant if needed, we obtain the bound \eqref{d2}.
\end{proof}

\subsection{Existence of optimal solutions}
We now consider the OCP \eqref{d1}–\eqref{e}. By virtue of Theorem \ref{Tam1} and the embedding \eqref{c5}, the cost functional \( J \) is well-defined. We define the control-to-state operator

\begin{align}\label{r9}
    G:L^{\infty}(\omega_T)\to W(0,T;\mathbb{V}) ,\quad v\mapsto G(v):=y
\end{align}

which maps each control \( v \in L^{\infty}(\omega_T) \) to the unique weak solution \( y \) of \eqref{e}. As a result, the control problem \eqref{d1}–\eqref{e} is equivalent to

\begin{align}\label{fun}
    \inf_{v\in U} J(u)=\frac{1}{2}\Vert G(v)(\cdot,T)-y^d\Vert^{2}_{L^{2}(\Omega)}+\frac{\alpha}{2}\Vert v\Vert^{2}_{L^{2}(\omega_T)}.
\end{align}

\begin{theorem}
    Let \( \alpha > 0 \), \( u \in U \), \( y_0, y^d \in L^{\infty}(\Omega) \), and \(f\in L^{\infty}(Q)\). Then there exists a control \( u \in U \) that solves the optimization problem \eqref{fun}, and hence also the original problem \eqref{d1}–\eqref{e}.
\end{theorem}

\begin{proof}
    Let \( v^n \in U \) be a minimizing sequence s.t.
    \begin{align*}
        \lim_{n\to\infty}J(v^n)=\inf_{v\in U} J(u).
    \end{align*}
    Denote \( y^n := G(v^n) \) as the associated state variables. From the boundedness of the cost functional, there exists a constant \( C > 0 \), independent of \( n \), satisfying
    \begin{align}\label{1r}
        \Vert y^n(\cdot, T)\Vert_{L^{2}(\Omega)}\le C \quad \text{and}\quad \Vert v^n\Vert_{L^{2}(\omega_T)}\le C.
    \end{align}
    Each \( y^n \) satisfies the state equation
    \begin{align}\label{en}
    (y^n)_{t}+(-\Delta)^{s}y^{n}(t,x)+\int_{0}^{t}\kappa(t,\tau,x)y^{n}(\tau,x)d\tau&=vy^{n}\chi_{\omega}+f\quad \text{in} \quad Q\nonumber\\
    y^{n}=&0\quad\quad\text{in}\quad\Sigma\\
    y^{n}(\cdot,0)&=y_0\quad\text{in}\quad\Omega\nonumber
\end{align}
According to Theorem \ref{t9}, equation \eqref{en} admits a unique solution \( y^n \in W(0,T;\mathbb{V}) \) that satisfies
\begin{align}\label{y0}
        -\int_{0}^{T}\langle \phi_{t},y^n\rangle_{\mathbb{V}^{*},\mathbb{V}}\,dt&+\int_{0}^{T}\mathcal{M}(y^n,\phi)\,dt+\int_{Q}\bigg(\int_{0}^{t}\kappa(t,\tau,x)y^{n}(\tau,x)\,d\tau\bigg)\phi\,dxdt\nonumber\\
      & -\int_{\omega_{T}}vy^n\phi\,dxdt=\int_{0}^{T}\langle f,\phi\rangle_{\mathbb{V}^*,\mathbb{V}}\,dt+\int_{\Omega}y_{0}\phi(0)dx,
    \end{align}
for all test functions \( \phi \in H(Q) \). Again using Theorem \ref{t9} and the bound \( \Vert v^n\Vert_{\infty}\le\max\{\vert m\vert,\vert M\vert\} \), there exists a constant \( C > 0 \) s.t.
\begin{align}\label{r}
        \Vert y^n\Vert_{(0,T;\mathbb{V})}\le Ce^{CT}[\Vert y_0\Vert_{L^{2}(\Omega)}+\Vert f\Vert_{L^{2}((0,T);\mathbb{V}^{*})}]
    \end{align}
From this estimate, we deduce
\begin{align}\label{r1}
      \Vert v^n  y^n\Vert_{L^{2}(\omega_T)} \le C \Vert y^n\Vert_{L^{2}(Q)} \le C \Vert y^n\Vert_{W(0,T;\mathbb{V})}\le Ce^{CT}[\Vert y_0\Vert_{L^{2}(\Omega)}+\Vert f\Vert_{L^{2}((0,T);\mathbb{V}^{*})}]
    \end{align}
Using \eqref{1r}–\eqref{r1}, we extract (up to a subsequence) functions \( \eta \in L^{2}(\Omega), u \in L^{2}(\omega_T), \beta \in L^{2}(\omega_T) \), and \( y \in W(0,T;\mathbb{V}) \) s.t.:
\begin{align}\label{y}
        v^n\rightharpoonup v\quad \text{weakly in} \quad L^{2}(\omega_T),
\end{align}
\begin{align}\label{y1}
        y^n(\cdot,T)\rightharpoonup \eta\quad \text{weakly in} \quad L^{2}(\Omega),
\end{align}
\begin{align}\label{y2}
        y^n\rightharpoonup y\quad \text{weakly in} \quad W(0,T;\mathbb{V}),
\end{align}
\begin{align}\label{y3}
        v^n y^n\rightharpoonup \beta\quad \text{weakly in} \quad L^{2}(\omega_T).
\end{align}
As \( U \subset L^{2}(\omega_T) \) is convex and closed, it is also weakly closed, so we get
\begin{align}\label{z1}
        u\in U.
\end{align}
Using the compactness result in Theorem 5.1 in \cite{jl1}, we obtain
\begin{align}\label{y4}
        y^n\to y\quad \text{strongly in} \quad L^{2}(Q).
\end{align}
From \eqref{y} and \eqref{y4}, we apply the weak-strong convergence principle to get
\begin{align}\label{y5}
        v^n y^n\rightharpoonup vy\quad \text{weakly in} \quad L^{1}(\omega_T).
\end{align}
Because \( v^n y^n \rightharpoonup \beta \) in \( L^2(\omega_T) \), and using the embedding \( L^\infty \hookrightarrow L^2 \hookrightarrow L^1 \), we deduce \( \beta = uy \). Hence,
\begin{align}\label{y6}
        v^n y^n\rightharpoonup uy\quad \text{weakly in} \quad L^{2}(\omega_T).
\end{align}
We can now pass to the limit in \eqref{y0} using \eqref{y2} and \eqref{y6} to get
\begin{align}\label{y7}
        -\int_{0}^{T}\langle \phi_{t},y\rangle_{\mathbb{V}^{*},\mathbb{V}}\,dt&+\int_{0}^{T}\mathcal{M}(y,\phi)\,dt+\int_{Q}\bigg(\int_{0}^{t}\kappa(t,\tau,x)y(\tau,x)\,d\tau\bigg)\phi\,dxdt\nonumber\\
      & -\int_{\omega_{T}}uy\phi\,dxdt=\int_{\Omega}y_{0}\phi(0)dx,
\end{align}
which confirms that \( y \) is the weak solution to \eqref{e} corresponding to the control \( v = u \).

Now, taking \( \phi \in W(0,T;\mathbb{V}) \) in \eqref{y0}, and applying Proposition \ref{p}, we integrate over time to find
\begin{align}\label{y8}
       & -\int_{0}^{T}\langle \phi_{t},y^n\rangle_{\mathbb{V}^{*},\mathbb{V}}\,dt+\int_{0}^{T}\mathcal{M}(y^n,\phi)\,dt+\int_{Q}\bigg(\int_{0}^{t}\kappa(t,\tau,x)y^{n}(\tau,x)\,d\tau\bigg)\phi\,dxdt\nonumber\\
      & -\int_{\omega_{T}}vy^n\phi\,dxdt=\int_{0}^{T}\langle f,\phi\rangle_{\mathbb{V}^*,\mathbb{V}}\,dt-\int_{\Omega}y^{n}(T)\phi(T)dx+\int_{\Omega}y_{0}\phi(0)dx,
\end{align}
Passing to the limit in \eqref{y8}, we arrive at
\begin{align*}
       &\int_{\Omega}\eta\phi(T)dx -\int_{0}^{T}\langle \phi_{t},y\rangle_{\mathbb{V}^{*},\mathbb{V}}\,dt+\int_{0}^{T}\mathcal{M}(y,\phi)\,dt\\
       &+\int_{Q}\bigg(\int_{0}^{t}\kappa(t,\tau,x)y(\tau,x)\,d\tau\bigg)\phi\,dxdt-\int_{\omega_{T}}uy\phi\,dxdt=\int_{0}^{T}\langle f,\phi\rangle_{\mathbb{V}^*,\mathbb{V}}\,dt+\int_{\Omega}y_{0}\phi(0)dx.
\end{align*}
Applying Proposition \ref{p} once more yields
\begin{align}\label{y10}
        &\int_{0}^{T}\langle y_t+(-\Delta)^{s}y+\int_{0}^{t}\kappa(t,\tau,x)y(\tau,x)d\tau-f,\phi\rangle_{\mathbb{V}^{*},\mathbb{V}}\,dt \\
        &=\int_{\omega_{T}}uy\phi\,dxdt-\int_{\Omega}y_{0}\phi(0)dx+\int_{\Omega}y_{0}\phi(0)dx-\int_{\Omega}(\eta-y(T))\phi(T)dx,\nonumber
\end{align}
leading to
\begin{align*}
        \int_{\Omega}(\eta-y(T))\phi(T)dx=0, \quad \forall\phi\in W(0,T;\mathbb{V}),
\end{align*}
and hence
\begin{align}\label{y11}
        \eta=y(\cdot,T)\quad \text{a.e. in} \quad \Omega.
\end{align}
Combining \eqref{y1}–\eqref{y11}, we conclude
\begin{align}\label{y12}
        y^n(\cdot,T)\rightharpoonup y(\cdot,T)\quad \text{weakly in} \quad L^{2}(\Omega).
\end{align}
Finally, by the lower semicontinuity of \( J \), and using \eqref{y12}, \eqref{y}, and \eqref{z1}, we obtain
\begin{align*}
        J(u)\le\lim\inf_{n\to\infty}J(v^n)=\inf_{v\in U} J(u),
\end{align*}
which completes the proof.
\end{proof}

The aim of this section is to derive the first-order necessary conditions for optimality and to provide a characterization of the optimal control. Prior to formulating the optimality system, we first investigate certain regularity properties of the control-to-state operator.

We introduce the operator
\begin{align}\label{z2}
    \begin{cases}
        \mathcal{G}: W(0,T;\mathbb{V}) \times L^{\infty}(\omega_T) \to L^2((0,T);\mathbb{V}^*) \times L^2(\Omega), \\
        \mathcal{G}(y,v) := \left( y_t + (-\Delta)^s y + \displaystyle\int_0^t \kappa(t,\tau,x)y(\tau,x)\,d\tau - vy\chi_{\omega},\; y(0) - y_0 \right).
    \end{cases}
\end{align}
With this definition, the state equation \eqref{e} can be compactly written as $\mathcal{G}(y,v) = (0,0)$.
 
\begin{lemma}\label{le1}
The mapping $\mathcal{G}$ defined in \eqref{z2} is of class $C^{\infty}$.
\end{lemma}
 
\begin{proof}
We decompose the first component $\mathcal{G}_1$ of $\mathcal{G}$ as
\[
\mathcal{G}_1(y,v)(\phi) = \mathcal{G}_{11}(y,v)(\phi) + \mathcal{G}_{12}(y,v)(\phi), \quad \forall \phi \in L^2((0,T);\mathbb{V}),
\]
where
\begin{align*}
\mathcal{G}_{11}(y,v)(\phi) &= \int_0^T \langle y_t, \phi \rangle_{\mathbb{V}^*, \mathbb{V}}\,dt + \int_0^T \mathcal{M}(y, \phi)\,dt \\
&\quad + \int_Q \left( \int_0^t \kappa(t,\tau,x) y(\tau,x)\,d\tau \right) \phi\,dx\,dt -\int_{0}^{T}\langle f,\phi\rangle_{\mathbb{V}^*,\mathbb{V}}\,dt,
\end{align*}
and
\[
\mathcal{G}_{12}(y,v)(\phi) = - \int_{\omega_T} v y \phi\,dx\,dt.
\]

The operator $\mathcal{G}_{11}$ is linear and continuous with respect to $y \in W(0,T;\mathbb{V})$, and $\mathcal{G}_{12}$ is bilinear and continuous in $(y,v) \in W(0,T;\mathbb{V}) \times L^\infty(\omega_T)$. Hence, both components are smooth, i.e., of class $C^{\infty}$. The second component of $\mathcal{G}$ is clearly $C^{\infty}$ as it is linear and continuous. 
\end{proof}

The proofs of the following lemmas are not provided here, as they can be derived using arguments analogous to those found in \cite{kdl, ct2}.
 
\begin{lemma}\label{le2}
The control-to-state operator $G: L^\infty(\omega_T) \to W(0,T;\mathbb{V})$, defined by $G(v) = y$, is of class $C^{\infty}$.
\end{lemma}

We now address the Lipschitz continuity of the mapping $G$.
 
\begin{lemma}\label{lem}
Let $v \in L^\infty(\omega_T)$ and $y_0 \in L^2(\Omega)$. Then there exists a constant $C = C(\|v_1\|_{L^\infty}, \|v_2\|_{L^\infty}, \|\kappa\|_{L^\infty}, T)$ s.t. the following estimate holds:
\begin{align}\label{lem1}
\|G(v_1) - G(v_2)\|_{W(0,T;\mathbb{V})} \leq C \left( \|y_0\|^2_{L^{\infty}(\Omega)} + \|f\|^2_{L^{\infty}(Q)} \right) \|v_1 - v_2\|^2_{L^2(\omega_T)}.
\end{align}
\end{lemma}

\begin{lemma}\label{lem4}
Let $G$ be the control-to-state operator defined as above. Then the directional derivative of $G$ at $v \in L^\infty(\omega_T)$ in the direction $u \in L^\infty(\omega_T)$ exists and is given by
\[
G'(v)u = \rho,
\]
where $\rho \in W(0,T;\mathbb{V})$ is the unique weak solution of the following linearized system:
\begin{align}\label{q3}
\begin{cases}
\rho_t + (-\Delta)^s \rho + \displaystyle\int_0^t \kappa(t,\tau,x)\rho(\tau,x)\,d\tau = (v\rho + uy)\chi_{\omega} & \text{in } Q, \\
\rho = 0 & \text{on } \Sigma, \\
\rho(\cdot,0) = 0 & \text{in } \Omega,
\end{cases}
\end{align}
where $y = G(v)$.
\end{lemma}

\begin{lemma}\label{lem5}
Under the assumptions of Lemma \ref{lem4}, the operator $G$ is twice continuously Fréchet differentiable from $L^\infty(\omega_T)$ into $W(0,T;\mathbb{V}) $. Moreover, the second derivative of $G$ at $v$ in the directions $w,h \in L^\infty(\omega_T)$ is given by
\[
G''(v)[w,h] = z,
\]
where $z \in W(0,T;\mathbb{V})$ is the unique weak solution of
\begin{align}\label{q5}
\begin{cases}
z_t + (-\Delta)^s z + \displaystyle\int_0^t \kappa(t,\tau,x)z(\tau,x)\,d\tau = (v z + h G'(v)w + w G'(v)h) \chi_{\omega} & \text{in } Q, \\
z = 0 & \text{on } \Sigma, \\
z(\cdot,0) = 0 & \text{in } \Omega.
\end{cases}
\end{align}
\end{lemma}
 
\subsection{First-order necessary optimality conditions}
We introduce the reduced cost functional defined by
\begin{align}\label{q6}
    \mathcal{J}(v) := J(G(v), v),
\end{align}
where \( J \) is given by \eqref{fun}.
 
\begin{remark}
Since both \( J \) and the control-to-state operator \( G \) are continuously Fréchet differentiable, so is \( \mathcal{J} \).
\end{remark}
 
We recall the notion of an \( L^{\infty} \)-local minimum for \eqref{fun}: a control \( u \in U \) is an \( L^{\infty} \)-local minimizer if there exists \( \varepsilon > 0 \) s.t.
\begin{align*}
    \mathcal{J}(u) \le \mathcal{J}(v) \quad \text{for all } v \in U \cap \mathbb{B}^{\infty}_{\varepsilon}(u),
\end{align*}
where \( \mathbb{B}^{\infty}_{\varepsilon}(u) \) denotes the open ball in \( L^{\infty}(\omega_T) \) centered at \( u \) with radius \( \varepsilon \).

The following result characterizes local minimizers:

We have the following  well-known results; see, for instance, \cite{ct2,kdl} for proofs.
\begin{theorem}(First-order necessary optimality conditions)
Let \( \alpha > 0 \) and \( y_0, y^d \in L^{\infty}(\Omega) \). If \( u \) is an \( L^{\infty} \)-local minimizer of \eqref{fun}, then
\begin{align}
    \mathcal{J}'(u)(v - u) \ge 0 \quad \text{for all } v \in U.
\end{align}
Furthermore, there exist functions \( y, q \in W(0,T; \mathbb{V}) \cap \mathbb{K} \) s.t. the triplet \( (y, u, q) \) satisfies the following system:
\begin{align}\label{i}
    y_t + (-\Delta)^s y + \int_0^t \kappa(t,\tau,x) y(\tau,x) \, d\tau &= u y \chi_{\omega}+f  \quad \text{in } Q, \nonumber\\
    y &= 0 \quad \text{on } \Sigma, \\
    y(\cdot,0) &= y_0 \quad \text{in } \Omega, \nonumber
\end{align}
\begin{align}\label{i1}
    -q_t + (-\Delta)^s q + \int_0^t \kappa(t,\tau,x) q(\tau,x) \, d\tau &= u q \chi_{\omega} + y(u) - y^d \quad \text{in } Q, \nonumber\\
    q &= 0 \quad \text{on } \Sigma, \\
    q(\cdot,T) &= 0 \quad \text{in } \Omega, \nonumber
\end{align}
and the variational inequality
\begin{align}\label{i2}
    \int_{\omega_T} (\alpha u + y(u) q)(v - u) \, dx dt \ge 0 \quad \text{for all } v \in U.
\end{align}
This implies the pointwise condition
\begin{align}\label{i3}
    u(x,t) =
    \begin{cases}
        m & \text{if } \alpha u(x,t) + y(x,t)q(x,t) > 0, \\
        \in [m, M] & \text{if } \alpha u(x,t) + y(x,t)q(x,t) = 0, \\
        M & \text{if } \alpha u(x,t) + y(x,t)q(x,t) < 0,
    \end{cases}
\end{align}
almost everywhere in \( \omega_T \).
\end{theorem}
 
\begin{remark}
The characterization \eqref{i3} is equivalent to the projection formula
\begin{align}\label{1z}
    u = \min\left( \max\left(m, -\frac{yq}{\alpha} \right), M \right) \quad \text{a.e. in } \omega_T.
\end{align}
\end{remark}
 
\begin{theorem}\cite{ct2,kmw}
    If the parameter \(\alpha\) is sufficiently large, then the OCP \eqref{i}-\eqref{i2}, \eqref{1z} admits a unique optimal solution.
\end{theorem}

\begin{proof}
    Suppose, by contradiction, that there exist two distinct optimal controls \(u\) and \(\bar{u}\) with corresponding state solutions \(y\) and \(\bar{y}\) to \eqref{i}, and adjoint states \(q\) and \(\bar{q}\) solving \eqref{i1}.

    Define the functions
    \[
    z := e^{-rt}(y - \bar{y}), \quad p := e^{-r(T-t)}(q - \bar{q}),
    \]
    where 
    \[
    r := \sigma + \Vert \kappa\Vert_{\infty} + 2 > 0, \quad \sigma := \max\{|m|, |M|\}.
    \]

    Then, \(z\) and \(p\) satisfy the following problems:
    \begin{align}\label{2i}
        \begin{cases}
            z_t + (-\Delta)^s z + \displaystyle \int_0^t e^{r(\tau - t)} \kappa(t,\tau,x) z(\tau,x) \, d\tau + r z = e^{-rt} (u - \bar{u}) y \chi_{\omega} + \bar{u} z \chi_{\omega} & \text{in } Q, \\
            z = 0 & \text{on } \Sigma, \\
            z(\cdot, 0) = 0 & \text{in } \Omega,
        \end{cases}
    \end{align}
    and
    \begin{align}\label{3i}
        \begin{cases}
            -p_t + (-\Delta)^s p + \displaystyle \int_t^T e^{r(\tau - t)} \kappa(t,\tau,x) p(\tau,x) \, d\tau + r p = e^{-r(T - t)} (u - \bar{u}) q \chi_{\omega} + \bar{u} p \chi_{\omega} + z & \text{in } Q, \\
            p = 0 & \text{on } \Sigma, \\
            p(\cdot, T) = 0 & \text{in } \Omega.
        \end{cases}
    \end{align}

    Multiplying \eqref{2i} by \(z\) and \eqref{3i} by \(p\), integrating over \(Q\), and applying integration by parts together with estimate \eqref{est}, we obtain
    \begin{align*}
        \Vert (-\Delta)^{s/2} z \Vert_{L^2(Q)}^2 + r \Vert z \Vert_{L^2(Q)}^2 &\leq \int_{\omega_T} (u - \bar{u}) y z \, dx dt + (\sigma + \Vert \kappa\Vert_{\infty}) \Vert z \Vert_{L^2(Q)}^2,
    \end{align*}
    and
    \begin{align*}
        \Vert (-\Delta)^{s/2} p \Vert_{L^2(Q)}^2 + r \Vert p \Vert_{L^2(Q)}^2 &\leq \int_{\omega_T} (u - \bar{u}) q p \, dx dt + (\sigma + \Vert \kappa\Vert_{\infty}) \Vert p \Vert_{L^2(Q)}^2 + \int_Q z p \, dx dt.
    \end{align*}

    Applying Young's inequality to the right-hand sides yields
    \begin{align*}
        \Vert (-\Delta)^{s/2} z \Vert_{L^2(Q)}^2 + r \Vert z \Vert_{L^2(Q)}^2 &\leq \frac{1}{2} \Vert z \Vert_{L^2(Q)}^2 + \frac{1}{2} \Vert y \Vert_{L^\infty(Q)}^2 \Vert u - \bar{u} \Vert_{L^2(\omega_T)}^2 + (\sigma + \Vert \kappa\Vert_{\infty}) \Vert z \Vert_{L^2(Q)}^2,
    \end{align*}
    and
    \begin{align*}
        \Vert (-\Delta)^{s/2} p \Vert_{L^2(Q)}^2 + r \Vert p \Vert_{L^2(Q)}^2 &\leq \frac{1}{2} \Vert p \Vert_{L^2(Q)}^2 + \frac{1}{2} \Vert q \Vert_{L^\infty(Q)}^2 \Vert u - \bar{u} \Vert_{L^2(\omega_T)}^2 \\
        & \quad + (\sigma + \Vert \kappa\Vert_{\infty}) \Vert p \Vert_{L^2(Q)}^2 + \frac{1}{2} \Vert z \Vert_{L^2(Q)}^2.
    \end{align*}

    Since \(r = \sigma + \Vert \kappa\Vert_{\infty} + 2\), combining these inequalities gives
    \[
    \Vert z \Vert_{L^2((0,T); \mathbb{V})}^2 + \Vert p \Vert_{L^2((0,T); \mathbb{V})}^2 \leq \Vert u - \bar{u} \Vert_{L^2(\omega_T)}^2 \left( \Vert y \Vert_{L^\infty(Q)}^2 + \Vert q \Vert_{L^\infty(Q)}^2 \right).
    \]

    Equivalently, this implies
    \begin{align}\label{1k}
        \Vert y - \bar{y} \Vert_{L^2((0,T); \mathbb{V})}^2 + \Vert q - \bar{q} \Vert_{L^2((0,T); \mathbb{V})}^2 \leq C \Vert u - \bar{u} \Vert_{L^2(\omega_T)}^2 \left( \Vert y_0 \Vert_{L^{\infty}(\Omega)}^2 + \Vert y^d \Vert_{L^{\infty}(\Omega)}^2 + \Vert f \Vert_{L^{\infty}(Q)}^2 \right),
    \end{align}
    for some constant \(C = C(\sigma, \Vert \kappa\Vert_{\infty}, T) > 0\).

    Next, from the control characterization \eqref{1z}, we estimate
    \[
    |u - \bar{u}|^2 \leq \frac{2}{\alpha^2} \left( |q - \bar{q}|^2 |y|^2 + |\bar{q}|^2 |y - \bar{y}|^2 \right),
    \]
    holding almost everywhere in \(\omega_T\).

    Integrating and applying norm inequalities, together with \eqref{d2} and \eqref{qo}, yields
    \begin{align}\label{2k}
        \Vert u - \bar{u} \Vert_{L^2(\omega_T)}^2 &\leq \frac{1}{\alpha^2} \left( \Vert y \Vert_{L^{\infty}(\Omega)}^2  \Vert q - \bar{q} \Vert_{L^2((0,T); \mathbb{V})}^2 + \Vert q\Vert_{L^{\infty}(\Omega)}^2 \Vert y - \bar{y} \Vert_{L^2((0,T); \mathbb{V})}^2 \right) \nonumber\\
        &\leq \frac{C}{\alpha^2} \left( \Vert y_0 \Vert_{L^{\infty}(\Omega)}^2 + \Vert y^d \Vert_{L^{\infty}(\Omega)}^2 + \Vert f \Vert_{L^{\infty}(Q)}^2 \right) \\
        &\quad \times \left( \Vert q - \bar{q} \Vert_{L^2((0,T); \mathbb{V})}^2 + \Vert y - \bar{y} \Vert_{L^2((0,T); \mathbb{V})}^2 \right). \nonumber
    \end{align}

    Combining \eqref{1k} and \eqref{2k} results in
    \[
    \Vert u - \bar{u} \Vert_{L^2(\omega_T)}^2 \leq \frac{C}{\alpha^2} \left( \Vert y_0 \Vert_{L^{\infty}(\Omega)}^2 + \Vert y^d \Vert_{L^{\infty}(\Omega)}^2 + \Vert f \Vert_{L^{\infty}(Q)}^2 \right)^2 \Vert u - \bar{u} \Vert_{L^2(\omega_T)}^2.
    \]

    Choosing \(\alpha\) sufficiently large so that
    \[
    \alpha^2 > C \left( \Vert y_0 \Vert_{L^2(\Omega)}^2 + \Vert y^d \Vert_{L^2(\Omega)}^2 + \Vert f \Vert_{L^2(0,T; \mathbb{V}^*)}^2 \right)^2,
    \]
    it follows that
    \[
    \Vert u - \bar{u} \Vert_{L^2(\omega_T)}^2 \leq \delta \Vert u - \bar{u} \Vert_{L^2(\omega_T)}^2,
    \]
    for some \(\delta < 1\), which implies
    \[
    \Vert u - \bar{u} \Vert_{L^2(\omega_T)} = 0.
    \]

    Hence, \(u = \bar{u}\) almost everywhere in \(\omega_T\). The uniqueness of the associated states and adjoint states \(y = \bar{y}\) and \(q = \bar{q}\) follows from the well-posedness of problems \eqref{i} and \eqref{i1}.

    Therefore, the OCP has a unique solution whenever \(\alpha\) is sufficiently large.
\end{proof}
 
\begin{remark}
\textbf{1.} The adjoint equation \eqref{i1} defines the adjoint state \( q \), which is introduced to characterize the optimal control.

\textbf{2.} Using the transformation \( t \mapsto T - t \), one can derive the estimate
\begin{align}\label{qo}
    &\|q\|^2_{L^{\infty}(Q)} +  \|q\|^2_{W((0,T);\mathbb{V})} \\
    &\le C(\|u\|_{\infty}, \|\kappa\|_{\infty}, T) \left[ \|y_0\|^2_{L^{\infty}(\Omega)}+\|f\|^2_{L^{\infty}(Q} + \|y^d\|^2_{L^{\infty}(\Omega)} \right]. \nonumber
\end{align}
\end{remark}
 
\begin{proposition}
Let \( u \in L^{\infty}(\omega_T) \). Then, the mapping \( u \mapsto q(u) \), where \( q(u) \) solves \eqref{i1}, is Lipschitz continuous from \( L^2(\omega_T) \) into \( L^2((0,T);\mathbb{V}) \cap L^{\infty}(Q) \). Specifically, for all \( u_1, u_2 \in L^{\infty}(\omega_T) \), there exists a constant \( C > 0 \), depending on the norms of \( u_1, u_2, y^d, y_0, f \), and \( T \), s.t.
\begin{align}
    \|q(u_1) - q(u_2)\|^2_{W(0,T;\mathbb{V})} \le C \|u_1 - u_2\|^2_{L^2(\omega_T)}.
\end{align}
\end{proposition}

\begin{proposition}(Twice Frechet differentiability)
Let \( u \in L^{\infty}(\omega_T) \), and let \( y = G(u) \) and \( q \) solve \eqref{i} and \eqref{i1}, respectively. Then the reduced cost functional \( \mathcal{J} \) is twice continuously Fréchet differentiable. Moreover, for all \( u, w, h \in L^{\infty}(\omega_T) \),
\begin{align}\label{1d}
    \mathcal{J}'(u)w &= \int_{\omega_T} (\alpha u + yq) w \, dx dt, \\
    \mathcal{J}''(u)[w, h] &= \int_{\omega_T} [hG'(u)w + wG'(u)h] q \, dx dt 
    + \int_Q (G'(u)w)(G'(u)h) \, dx dt 
    + \alpha \int_{\omega_T} hw \, dx dt. \label{2d}
\end{align}
\end{proposition}
 
\begin{lemma}
For any \( u \in L^{\infty}(\omega_T) \), the linear map \( v \mapsto \mathcal{J}'(u)v \), defined initially on \( L^{\infty}(\omega_T) \), extends continuously to the whole space and is given explicitly by \eqref{1d}.
\end{lemma}
 
\subsection{Second order optimality conditions}
In this section, we derive second-order optimality conditions for the control problem \eqref{fun}. The following results will be instrumental:
 
\begin{lemma}[\cite{kdl,kmw}]
Let \( G \) denote the control-to-state operator defined in \eqref{r9}. Then for any \( u \in L^{\infty}(\omega_T) \), the linear mapping \( v \mapsto G'(u)v \) extends continuously from \( L^2(\omega_T) \) to \( W(0,T;\mathbb{V}) \). Moreover, there exists a constant \( C := C(\|u\|_{\infty}, \|\kappa\|_{\infty}, T) > 0 \) s.t.
\begin{align}
    \|G'(u)v\|_{W(0,T;\mathbb{V})} \le C  \big[\|y_0\|_{L^{\infty}(\Omega)}+\|f\|_{ L^{\infty}(Q)} \big]\|v\|_{L^2(\omega_T)},
\end{align}
for all \( v \in L^2(\omega_T) \).
\end{lemma}

\begin{lemma}[\cite{ct2,kmw}]
Let \( u \in L^{\infty}(\omega_T) \). Then the bilinear form \( (w,h) \mapsto \mathcal{J}''(u)[w,h] \) extends continuously to a mapping
\[
\mathcal{J}''(u) : L^2(\omega_T) \times L^2(\omega_T) \to \mathbb{R},
\]
given explicitly by \eqref{2d}.
\end{lemma}
 
We now introduce auxiliary notions adapted from \cite{tf}. For a given \( \tau \ge 0 \), define the set of strongly active constraints as
\begin{align*}
    A_{\tau}(u) := \{ (t,x) \in \omega_T : |\alpha u(x,t) + y(u)(x,t)q(x,t)| > \tau \}.
\end{align*}
The associated \( \tau \)-critical cone at control \( u \) is defined as
\begin{align}
    C_{\tau}(u) := \{ v \in L^{\infty}(\omega_T) : \text{ satisfies condition } \eqref{1i} \},
\end{align}
that is, for almost every \( (t,x) \in Q \),
\begin{align}\label{1i}
    \begin{cases}
        v(x,t) \ge 0 & \text{if } u(x,t) = m \text{ and } (x,t) \notin A_{\tau}(u), \\
        v(x,t) \le 0 & \text{if } u(x,t) = M \text{ and } (x,t) \notin A_{\tau}(u), \\
        v(x,t) = 0 & \text{if } (x,t) \in A_{\tau}(u).
    \end{cases}
\end{align}

In what follows, we use the notation \( \mathcal{J}''(u)v^2 := \mathcal{J}''(u)[v,v] \).

\begin{theorem}(Second-order necessary optimality conditions, \cite{ct2,kmw})
Let \( u \in U \) be an \( L^{\infty} \)-local minimizer of \eqref{fun}. Then
\begin{align}
    \mathcal{J}''(u)v^2 \ge 0 \quad \text{for all } v \in C_{\tau}(u).
\end{align}
\end{theorem}

\begin{theorem}[\cite{kdl,mkw}]
Let \( u \in U \) be an \( L^{\infty} \)-local control satisfying the first-order condition \eqref{1d}. Then:

\textbf{(a)} The reduced cost functional \( \mathcal{J} : L^{\infty}(\omega_T) \to \mathbb{R} \) is of class \( C^{\infty} \). Moreover, there exist continuous extensions
\begin{align}
    \mathcal{J}'(u) \in \mathcal{L}(L^2(\omega_T), \mathbb{R}), \quad 
    \mathcal{J}''(u) \in \mathcal{B}(L^2(\omega_T), \mathbb{R}).
\end{align}

\textbf{(b)} For any sequence \( \{(u_k, v_k)\}_{k=1}^{\infty} \subset U \times L^2(\omega_T) \) s.t. \( u_k \to u \) strongly in \( L^2(\omega_T) \) and \( v_k \rightharpoonup v \) weakly in \( L^2(\omega_T) \), we have
\begin{align}
    \mathcal{J}'(u)v = \lim_{k \to \infty} \mathcal{J}'(u_k)v_k,
\end{align}
and
\begin{align}
    \mathcal{J}''(u)v^2 \le \liminf_{k \to \infty} \mathcal{J}''(u_k)v_k^2.
\end{align}
In particular, if \( v = 0 \), then
\begin{align}
    \alpha \liminf_{k \to \infty} \|v_\kappa\|_{L^2(\omega_T)}^2 \le \liminf_{k \to \infty} \mathcal{J}''(u_k)v_k^2.
\end{align}
\end{theorem}

\begin{theorem}(Second-order sufficient optimality conditions, \cite{ct2,kdl,kmw})
Let \( y_0, y^d \in L^{\infty}(\Omega) \), and let \( u \in U \) be a control satisfying the variational inequality \eqref{i2} and
\begin{align}
    \mathcal{J}''(u)v^2 > 0 \quad \text{for all } v \in C_0(u) \setminus \{0\}.
\end{align}
Then, there exist constants \( \gamma > 0 \) and \( \beta > 0 \) s.t.
\begin{align}
    \mathcal{J}(v) \ge \mathcal{J}(u) + \frac{\beta}{2} \|v - u\|_{L^2(\omega_T)}^2 \quad \text{for all } v \in U \cap \overline{\mathbb{B}^{2}_{\gamma}}(u),
\end{align}
where \( \overline{\mathbb{B}^{2}_{\gamma}}(u) \) denotes the closed ball in \( L^2(\omega_T) \) centered at \( u \) with radius \( \gamma \).
\end{theorem}





\end{document}